\documentclass[11pt]{article}
\usepackage[utf8]{inputenc}

\usepackage{comment}
\usepackage{amssymb,amsmath,amsthm,graphicx,xcolor,mathtools,enumerate,mathrsfs}
\usepackage{fullpage,enumitem}
\usepackage{thmtools}
\usepackage{thm-restate}
\usepackage{hyperref}
\usepackage{indentfirst}

\usepackage{tikz}
\usetikzlibrary{math}
\usetikzlibrary{snakes,calc,decorations.pathmorphing,through}
\tikzset{snake it/.style={decorate, decoration=snake}}

\usepackage{cleveref}
\theoremstyle{plain}
\newtheorem{theorem}{Theorem}[section]
\crefname{theorem}{Theorem}{Theorems}

\crefname{proposition}{Proposition}{Propositions}

\newtheorem{corollary}[theorem]{Corollary}
\crefname{corollary}{Corollary}{Corollaries}

\newtheorem{lemma}[theorem]{Lemma}
\crefname{lemma}{Lemma}{Lemmas}

\newtheorem{conjecture}[theorem]{Conjecture}
\crefname{conjecture}{Conjecture}{Conjectures}

\crefname{problem}{Problem}{Problem}

\crefname{claim}{Claim}{Claims}

\crefname{setup}{Setup}{Setups}

\crefname{fact}{Fact}{Facts}

\crefname{algorithm}{Algorithm}{Algorithms}

\crefname{remark}{Remark}{Remarks}

\crefname{example}{Example}{Examples}

\theoremstyle{definition}
\newtheorem{definition}[theorem]{Definition}
\crefname{definition}{Definition}{Definitions}

\crefname{construction}{Construction}{Constructions}

\newtheorem{question}[theorem]{Question}
\crefname{question}{Question}{Questions}

\numberwithin{equation}{section}

\newcommand{\w}[1]{\widehat{#1}}

\DeclareMathOperator{\sub}{Sub}

\renewcommand{\int}[1]{\mathop{\mkern 0mu\mathrm{int}}\nolimits(#1)}

  \title{Separating edges by linearly many subdivisions}
\author{George Kontogeorgiou$^\dagger$\and Mat\'ias Pavez-Sign\'e$^{\dagger}$\thanks{Departamento de Ingenier\'ia Matem\'atica, Universidad de Chile, Beauchef 851, Santiago, Chile.} \and Maya Stein$^{\dagger*}$\and S Taruni$^{\dagger}$ \and Ana Trujillo-Negrete\thanks{Centro de Modelamiento Matem\'atico (CNRS IRL2807), Beauchef 851, Santiago, Chile. 
\newline All authors acknowledge the support of Centro de Modelamiento Matemático (CMM) BASAL fund FB210005 for center of excellence from ANID-Chile, and of the MSCA-RISE-2020-101007705 project {\it RandNET}. Additionally, GK acknowledges support by ANID-FONDECYT Postdoctorado Grant No. 3250479. MPS acknowledges support  by ANID-FONDECYT Regular Grant No.\ 1241398. MS acknowledges support  by FONDECYT Regular Grant 1221905 and  by FAPESP  Grant 2023/03167-5. AT acknowledges support by ANID-FONDECYT Postdoctorado Grant No. 3220838.
}}
\date{}
\begin{document}

 \maketitle
\begin{abstract}
    We prove that for any two graphs $G$ and $H$, the edges of $G$ can be strongly separated by a collection of linearly many subdivisions of $H$ and single edges. This confirms a conjecture of Botler and Naia. 
\end{abstract}
\section{Introduction}
For a set $X$ and a pair of distinct elements $x,y\in X$, we say that a family $\mathcal{F}$ of subsets of $X$ \textit{separates} $x$ \emph{from} $y$ if there is a set $F\in\mathcal F$ such that $x\in F$ but $y\not\in F$. We say that $\mathcal{F}$ \emph{strongly separates} $X$ if, for all distinct elements $x,y\in X$, $\mathcal F$ separates  $x$ from $y$ and $y$ from $x$. Similarly, $\mathcal{F}$ \emph{weakly separates} $X$ if, for every pair of elements of $X$, $\mathcal F$ separates at least one of them from the other.  We may also say that $\mathcal{F}$ is a \emph{strongly}/\emph{weakly} \emph{separating system} of $X$, respectively. 

The study of separating set systems was initiated in 1961 by R\'enyi~\cite{renyi1961random}, who was interested in determining how  many randomly chosen subsets of $[n]:=\{1,\ldots,n\}$ are enough to separate $[n]$. Many other authors subsequently contributed to this topic (see e.g.~\cite{katona1966separating, kundgen2001minimal,ramsay1996minimal,wegener1979separating}). 
Of particular interest is the case that the ground set $X$ is the edge set of some graph $G$ and the family $\mathcal{F}$ is a collection of subgraphs of $G$, in which case we say that $\mathcal F$ is a strongly/weakly separating system of $G$. Separating systems of graphs have found  applications in the field of network monitoring and fault detection and have therefore been studied extensively (see e.g.~\cite{ahuja2009single,harvey2007non,honkala2003cycles,tapolcai2012link,tapolcai2009monitoring,zakrevski1998fault}). 

Given a graph $G$, a family $\mathcal F$ of subgraphs of $G$ is a strongly/weakly separating path system of $G$ if $\mathcal F$ is a strongly/weakly separating system of $G$ that consists of paths.  Motivated by a question of Katona (see~\cite{letzter2024separating}), Falgas-Ravry, Kittipassorn, Korandi, Letzter and Narayanan~\cite{falgas-ravryseaparatingpath} initiated the study of weakly separating path systems and conjectured that any $n$-vertex graph admits one of size $O(n)$.
Later, Balogh, Csaba, Martin and Pluh\'{a}r~\cite{balogh2016path} conjectured similarly for strong separation. The first general result in this direction was obtained by Letzter~\cite{letzter2024separating}, who used ideas from sublinear expanders to show that every $n$-vertex graph can be strongly separated with $O(n\log^\star n)$ paths. Soon afterwards, Bonamy, Botler, Dross, Naia and Skokan~\cite{bonamy2023separating} found an elegant argument showing that any $n$-vertex graph admits a strongly separating path system of size at most $19n$, thus settling both conjectures. 

A natural next step in this area is to study separating systems consisting of other kinds of subgraphs. Given a class of graphs $\Sigma$, we say that a graph $G$ has a \textit{(strongly) separating $\Sigma$-system} if there is a family $\mathcal F$ of subgraphs of $G$ that (strongly) separates the edges of $G$ and every element of $\mathcal F$ is a member of the class $\Sigma$.  We are interested in the following question. 

\begin{question}\label{question:1}
     For which graph classes $\Sigma$ does every $n$-vertex graph $G$ admit a separating $\Sigma$-system of linear size?
\end{question}

Note that any candidate answer $\Sigma$ for Question \ref{question:1} must contain $K_2$, otherwise either matchings or stars cannot have a separating $\Sigma$-system of any size. Also, $G$ may be assumed to be connected. 

We say that a graph $H'$ is a \emph{subdivision} of a graph $H$ (or \emph{$H$-subdivision}) if $H'$ is obtained from $H$ by replacing some edges of $H$ with pairwise internally vertex-disjoint paths, called \emph{branch paths}. The vertices of $H'$ that are not internal vertices of branch paths are its \textit{branch vertices}. We use the notation $v_h$, $h\in V(H)$ for the branch vertices and $P_e$, $e\in E(H)$ for the branch paths. We say that $H'$ is \emph{$\ell$-balanced} if all of its branch paths are of length $\ell$, and \emph{$\ell$-almost-balanced} if at most one is not. We also use these adjectives to describe families of $H$-subdivisions the elements of which are $\ell$-(almost)-balanced for some common $\ell$.

Given a graph $H$, let $\sub(H)$ denote the class consisting  of $K_2$ and all the subdivisions of $H$. From the work of Bonamy et al.~\cite{bonamy2023separating}, we know that every graph has a $\sub(K_2)$-separating system of linear size. Extending this result, Botler and Naia~\cite{botler2024separating} proved that both $\sub(K_3)$ and $\sub(K_4)$ satisfy Question~\ref{question:1}, with respective bounds $41n$ and $82n$. They furthermore conjectured that their results were indicative of a more general fact. 

\begin{conjecture}[\cite{botler2024separating}]\label{conjecture} For every graph $H$ with at least one edge, there is a constant $C_H>0$ such that every $n$-vertex graph admits a separating $\sub(H)$-system of size at most $C_Hn$.
\end{conjecture}



In this paper, we prove this conjecture. 
\begin{theorem}
\label{thm:main}
There exists an absolute constant $C>0$ such that the following holds for every graph $H$: every $n$-vertex graph $G$  admits a separating $\sub(H)$-system of size at most $C|H|^2n$.   
\end{theorem}

In fact, if $G$ is additionally 3-connected and has average degree bounded below by a certain constant, then we obtain a much stronger result:

\begin{theorem}\label{3con}
    There exists an absolute constant $C>0$ such that the following holds for every graph $H$: every 3-connected, $n$-vertex graph $G$ that has average degree at least $C|H|^2$ admits an $\ell$-almost-balanced separating $\sub(H)$-system of size at most $984n$. 
\end{theorem}



We then turn our sights to Question \ref{question:1} in its full generality. We remark that, even if we only aim for size $o(n^2)$, we must restrict our attention to classes that contain an infinite number of bipartite graphs. For an appropriate meaning of ``large" we prove:

 \begin{theorem}\label{bipartitestuff}
     If $\Sigma$ is a large class of bipartite graphs, then every graph on $n$ vertices has a separating $\Sigma$-system of size $o(n^2)$.
 \end{theorem}

 We conclude this paper by introducing and exploring the more general notion of $H$\emph{-separating} $\Sigma$\emph{-systems} for arbitrary graphs $H$.

\section{Proofs}

For a graph $G$, the vertex set of $G$ is denoted $V(G)$ and the edge set of $G$ is $E(G)$. A \emph{path} $P$ in a graph $G$ is a sequence of pairwise distinct vertices $v_1,\ldots, v_m$ such that any two consecutive vertices in the sequence form an edge. We say that $v_1$ and $v_m$ are the \emph{endpoints} of $P$ and $v_2,\ldots, v_{m-1}$ are the \emph{internal vertices} of $P$. We say a path $P$ is a $u,v$\emph{-path} if $P$ has $u$ and $v$ as endpoints. 

\subsection{Proof of Theorem~\ref{thm:main}}
 A \textit{tree decomposition} of a connected graph $G$ on at least two vertices is a pair $(T,\mathcal V)$ where $T$ is a tree and $\mathcal V=\{V_t:t\in V(T)\}$ is a collection of subsets of $V(G)$, called \textit{bags}, such that 1) for every edge $xy$ of $G$, there is a bag $V_i\in\mathcal V$ that contains both $x$ and $y$, and 2) for every vertex $v\in V(G)$, the set $\{V_i:v\in V_i\}$ induces a non-empty subtree of $T$.  The \emph{adhesion set} of two distinct bags $V_i, V_j\in\mathcal V$ is the set of shared vertices $V_i\cap V_j$, and the \emph{adhesion} of the tree decomposition $(T,\mathcal V)$ is the maximum size of an adhesion set. The \emph{torso} of a bag $V_i\in \mathcal V$ is the subgraph $G_i$ induced by $V_i$ together with a spanning complete graph at each adhesion set in $V_i$. The edges that are added to a bag to form its torso are called \emph{virtual edges} to distinguish them from the \emph{real edges}, that is, the elements of $E(G)$.

 Given a separator $S$ of a connected graph $G$, a \emph{separation} induced by $S$ is a pair of connected subgraphs $(A,B)$ such that $G=A\cup B$ and $S=V(A\cap B)$, where $A$ and $B$ are the \emph{sides} of the separation. We say that two separations $(A,B)$ and $(C,D)$ are \emph{nested} if, for some choice of distinct $X,X'\in \{A,B\}$ and distinct $Y,Y'\in \{C,D\}$, we have $X\subseteq Y$ and $Y'\subseteq X'$, and we say that two separators are nested if they induce nested separations (this does not depend on the choice of separation). A 2-separator is \emph{totally nested} if it is nested with respect to every other 2-separator in the graph.   

\begin{definition} 
A \emph{Tutte decomposition} of a connected graph $G$ is a tree decomposition of adhesion at most two such that 

\begin{itemize}
    \item the torso of each bag is either 3-connected, a cycle, or a single real edge;
    \item if $\{u,v\}$ is an adhesion set such that $uv$ is a virtual edge that induces the separation $(A,B)$ on $G$, then in each of $A$ and $B$ there exists a path from $u$ to $v$.
\end{itemize}
\end{definition}

 The torsos of the bags of a Tutte decomposition are called its \emph{components}.

\begin{figure}[h!]
    \centering
    \includegraphics[scale=.32]{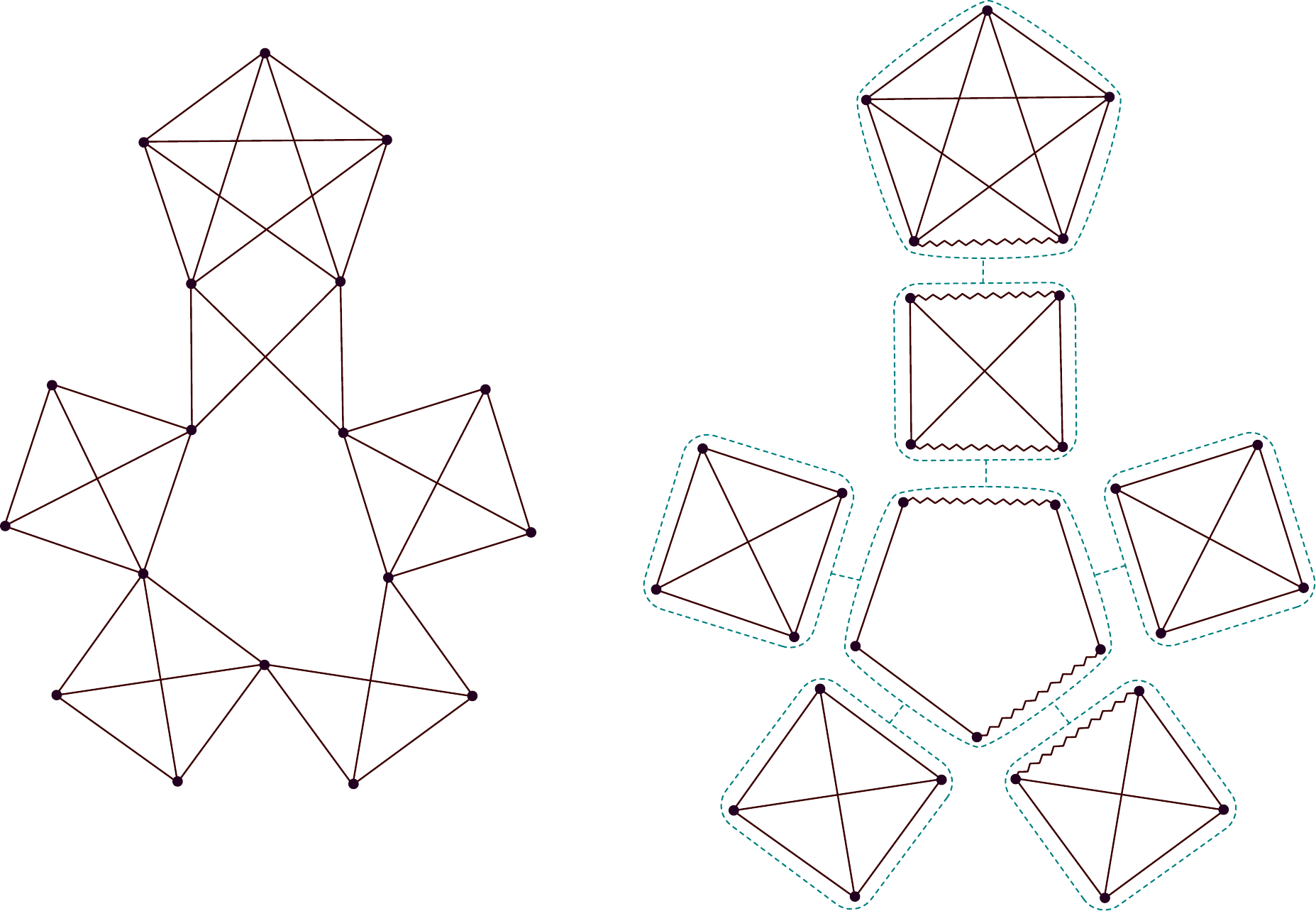}
       \caption{A graph $G$ together with its Tutte decomposition, each  zigzag representing a virtual edge.}
    \label{fig:tutte-decomp} 
\end{figure}

\begin{theorem}[\cite{tutte1966connectivity}]\label{thm:tutte}Every connected graph has a Tutte decomposition, which can be constructed by taking the adhesion sets of adjacent bags to be the totally nested separators of size at most two.    
\end{theorem} 

We firstly show that $3$-connected graphs can be separated with linearly many $H$-subdivisions. We then conclude the proof by taking a Tutte decomposition of an arbitrary connected graph, separating the edges of each of its components with linearly many members of $\sub(H)$, and finally modifying and combining those separating $\sub(H)$-systems into one.

\begin{lemma}\label{lemma:3-connected}There exists an absolute constant $C_{\ref{lemma:3-connected}}>0$ such that the following holds for every graph $H$ with at least one edge. If $G$ is a $3$-connected graph on $n$ vertices, then $G$ admits a separating $\sub (H)$-system of size at most $C_{\ref{lemma:3-connected}}|H|^2n$.\end{lemma}
The proof of Lemma~\ref{lemma:3-connected} will be postponed to Section~\ref{sec:third-approach}. We now prove Theorem~\ref{thm:main}.
\begin{proof}[Proof of Theorem~\ref{thm:main}]Let $C=3C_{\ref{lemma:3-connected}}$. For a given graph $G$, consider a Tutte decomposition of $G$ into components $G_1,\ldots, G_k$. For each $i\in [k]$, $G_i$ is either 3-connected, a cycle, or a single real edge, so it has a separating $\sub(H)$-system $\mathcal{F}^{ virtual}_i$ of size at most $C_{\ref{lemma:3-connected}}|H|^2|G_i|$ (here we use that, if $G_i$ is a cycle or an edge, then it can be separated with at most $|G_i|$ edges, and we can take $C_{\ref{lemma:3-connected}}\geq 1>\frac{1}{|H|^2}$). For each virtual edge $u_sv_s$ in an $H$-subdivision $H^{virtual}\in\mathcal{F}_i^{virtual}$, let $(A_s,B_s)$ be the separation induced on $G$ by $\{u_s,v_s\}$ and let $A_s$ be the side that does not contain $V(G_i)$. There exists a path $R_s$ from $u_s$ to $v_s$ in $A_s$. 

Since the adhesion sets are nested separators of $G$, the side $B_s$, which contains $V(G_i)$ and therefore all adhesion sets $\{u_{s'},v_{s'}\}$ for $s'\neq s$, also contains all the sides $A_{s'}$ and hence all the paths $R_{s'}$. That is, for all pairs $s\neq s'$, the paths $R_s$ and $R_{s'}$ are vertex-disjoint. Therefore, by substituting in $H^{virtual}$ each virtual edge $u_sv_s$ with the path $R_s$, we produce an $H$-subdivision $H'$ consisting of real edges and so $H'\subseteq G$; see Figure~\ref{fig:concatenation}. Let $\mathcal{F}_i$ denote the $\sub(H)$-system consisting of all the $H$-subdivisions obtained from $\mathcal{F}_i^{virtual}$ and of all the single edges in $\mathcal{F}_i^{virtual}$ that are real. 

\begin{figure}[h!]
    \centering
    \includegraphics[scale=.33]{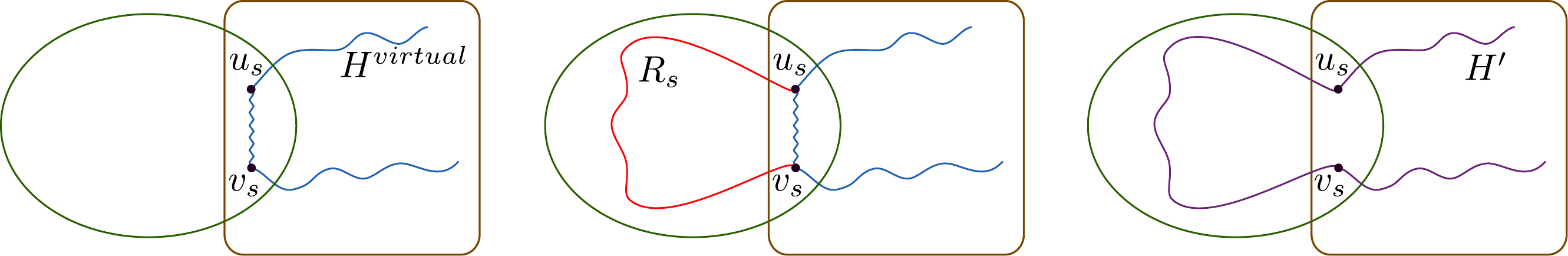}
    \caption{Modifying $H^{virtual}$ to obtain $H'$.}
    \label{fig:concatenation}
\end{figure}
We let $\mathcal{F}:=\bigcup_{i=1}^k\mathcal{F}_i$ and claim that $\mathcal F$ strongly separates the edges of $G$. Indeed, let us show this for two arbitrary distinct $e, e'\in E(G)$. If $e$ is itself an element of $\mathcal{F}$, we are done. Otherwise, if $e$ and $e'$ are in the same component $G_i$, then there is an $H$-subdivision $H^{virtual}\in \mathcal{F}^{virtual}_i$ such that $H^{virtual}$ separates $e$ from $e'$. So, the $H$-subdivision $H'\in\mathcal F_i$ obtained from $H^{virtual}$ separates $e$ from $e'$, so again we are done. 

So we can assume that $e\in E(G_i)$ and $e'\in E(G_j)$ for distinct $i,j\in [k]$, and let $\{u,v\}$ be the unique adhesion set of $G_i$ that separates $V(G_i)$ from $V(G_j)$. Take $H^{virtual}\in \mathcal{F}^{virtual}_i$ that separates $e$ from $uv$. Then the $H$-subdivision in $\mathcal F_i$ obtained from $H^{virtual}$ does not intersect $G_{j}$, hence separates $e$ from $e'$.

Let us now bound the size of the family $\mathcal{F}$. We enumerate the components so that, for each $i\in [k]$, $\mathbb{G}_i:=\bigcup_{j=1}^i G_j$ induces a subtree of the Tutte decomposition. Note that, since $|\mathbb{G}_i\setminus \mathbb{G}_{i-1}|\geq 1$ for each $2\le i\le k$, we have $k\leq n$. Moreover, the last component of $\mathbb{G}_i$ intersects $\mathbb{G}_{i-1}$ in at most two vertices, hence $\sum_{i=1}^k |G_i|\leq 3n$. Therefore, we have 
\[|\mathcal{F}|\leq \sum_{i=1}^k|\mathcal{F}_i|\leq \sum_{i=1}^k|\mathcal{F}_i^{virtual}|\leq C_{\ref{lemma:3-connected}}|H|^2\sum_{i=1}^k |G_i|\leq C|H|^2n.\]    \end{proof}

\subsection{Proof of Lemma~\ref{lemma:3-connected}}\label{sec:third-approach}
We have a $t$-vertex graph $H$ and an $n$-vertex graph $G$ that is 3-connected, and we wish to find a separating $\sub(H)$-system of size $O(n)$. As a first step, we find a large subdivision using a result of Gil Fernandez, Hyde, Liu, Pikhurko and Wu \cite{fernandez2023disjoint}, and of Luan, Tang, Wang and Yang \cite{luan2023balanced}. 
\begin{theorem}[\cite{fernandez2023disjoint,luan2023balanced}]There is a positive constant $c$ such that every graph with average degree at least $ct^2$ contains a balanced $K_t$-subdivision.  \end{theorem}

We will also need the result by Botler and Naia~\cite{botler2024separating} that was mentioned in the Introduction: 

\begin{theorem}[\cite{botler2024separating}]\label{thm:separating:cycles}
    Every graph on $n$ vertices has a separating $\sub(K_3)$-system of size $\leq 41n$.
\end{theorem}
\begin{proof}[Proof of Lemma~\ref{lemma:3-connected}]
    
If $G$ has average degree $d\le c(4t+8)^2$, then $E(G)$ is already an edge-separating system of size at most $\frac{d}{2}n\le \frac{c(4t+8)^2}{2}n$. Therefore, we may assume that $G$ has large average degree and so contains an $\ell$-balanced $K_{4t+8}$-subdivision $K$ for some $\ell\in\mathbb{N}$. Let us partition the branch vertices of $K$ into 4 parts of the same size to find vertex-disjoint $K_{t+2}$-subdivisions $K^1,K^2,K^3,K^4$.  We will construct four $\sub(H)$-systems $\mathcal{F}_r$, the union of which separates $E(G)$. 

\begin{figure}[h!]
    \centering
    \includegraphics[scale=.28]{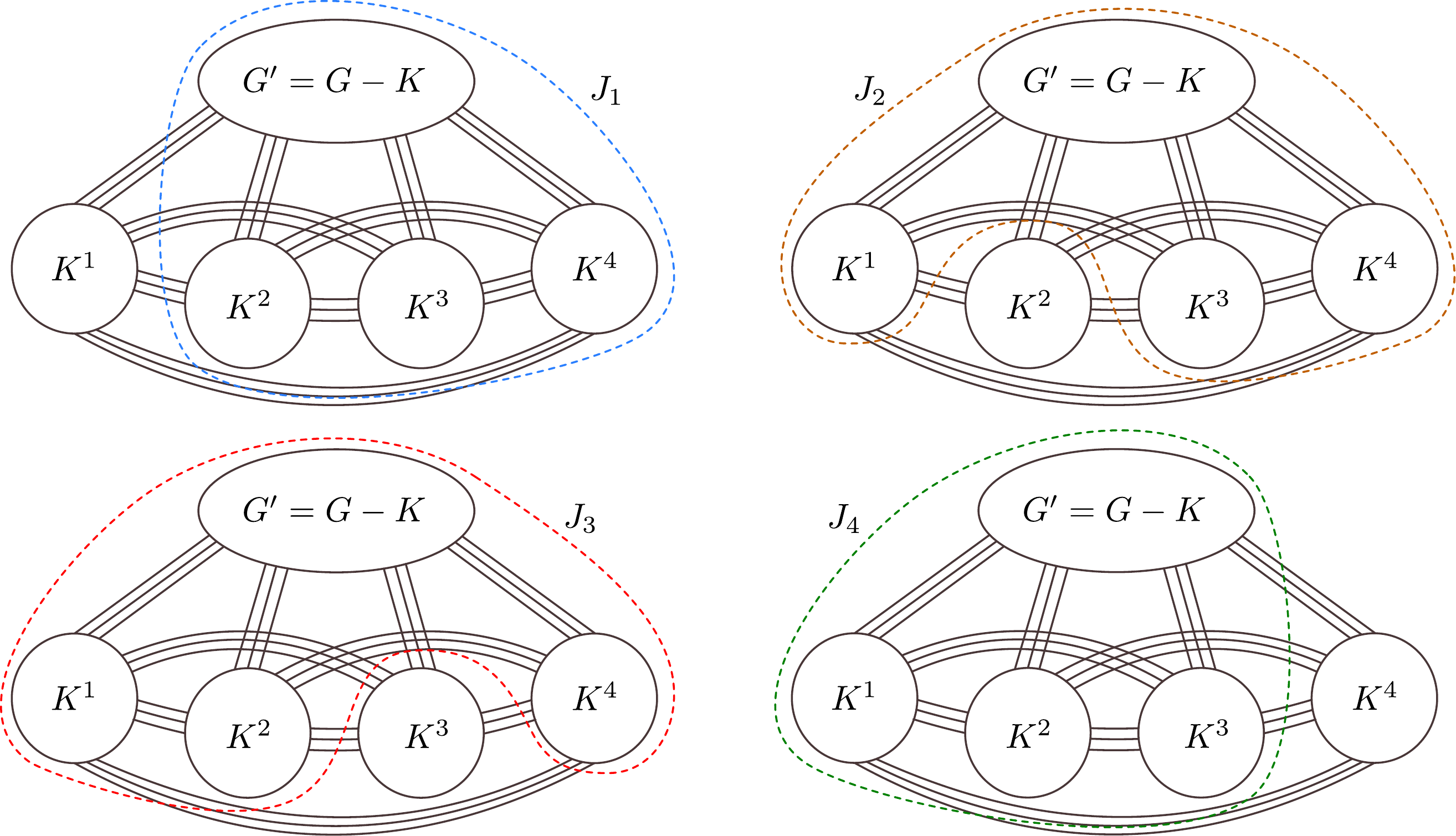}
    \caption{The subgraphs $J_1$, $J_2$, $J_3$, and $J_4$.}
    \label{fig:j-sets}
\end{figure}

For each $r\in [4]$, let $\mathcal{E}_r$ be the set of edges that are not incident to $K^r$ and let $J_r$ be the graph induced by $\mathcal{E}_r$ (see Figure~\ref{fig:j-sets}). We fix an $r\in [4]$ and we use Theorem~\ref{thm:separating:cycles} to find a cycle separating system $\mathcal{C}_r$  for $\mathcal{E}_r$ of size at most $41|J_r|$. We then define the following family $\mathcal F_r$ of subgraphs of $J_r$. Firstly, we add to $\mathcal{F}_r$ every element of $\mathcal{C}_r$ that is an edge, and then, for each cycle $C\in\mathcal C_r$, we produce six $H$-subdivisions in $G$ as follows:
	\begin{itemize}
		\item For each $i\in [3]$, we find vertices $x_i\in V(C)$ and $y_i\in V(K^r)$ and an $x_i,y_i$-path $P_i$ that intersects $C$ and $K^r$ only at $x_i$ and $y_i$, respectively, and such that $P_1$, $P_2$ and $P_3$ are pairiwse vertex-disjoint. As $G$ is $3$-connected, these paths exist by Menger's Theorem.
		\item For distinct $i, j, k \in [3]$, let $P_{i,j}$ denote the subpath of $C$ that joins $x_i$ to $x_j$ while avoiding $x_k$, and let $P'_{i,j}:=P_{i,k}P_{k,j}$.
        \item For each $i\in [3]$, let $Q_i$ be the branch path that contains $y_i$ (we note that these paths may coincide). For distinct $i,j\in[3]$, let $u_{i,j}$ and $v_{i,j}$ be distinct branch vertices incident to $Q_{i}$ and $Q_{j}$, respectively, and let
		\begin{align*}
		\w{P}_{i,j}:=& u_{i,j}Q_{i}y_iP_ix_iP_{i,j}x_jP_jy_jQ_{j}v_{i,j}, \\ 
		\w{P}'_{i,j}:=& u_{i,j}Q_{i}y_iP_ix_iP_{i,j}'x_jP_jy_jQ_{j}v_{i,j}.
		\end{align*}
        See Figure~\ref{fig:construction} for the construction of the paths $\w{P}_{i,j}$ and $\w{P}'_{i,j}$. 
    \item Let $f\in E(H)$ be fixed. For distinct $i,j\in[3]$, let $H_{i,j}^-$ be a subdivision of $H-f$ in $K^r$ such that the branch path corresponding to the edge $f$ has endpoints $u_{i,j}$ and $v_{i,j}$, and such that, if $Q_i$ and $Q_j$ have ends other than $u_{i,j}$ and $v_{i,j}$, these ends are not in $H_{i,j}^-$. We add to $\mathcal{F}_r$ the $H$-subdivisions $H_{i,j}:=H_{i,j}^-+\w{P}_{i,j}$ and $H_{i,j}':=H_{i,j}^-+\w{P}_{i,j}'$. 
    \end{itemize}
    \begin{figure}[h!]
    \centering
    \includegraphics[scale=.36]{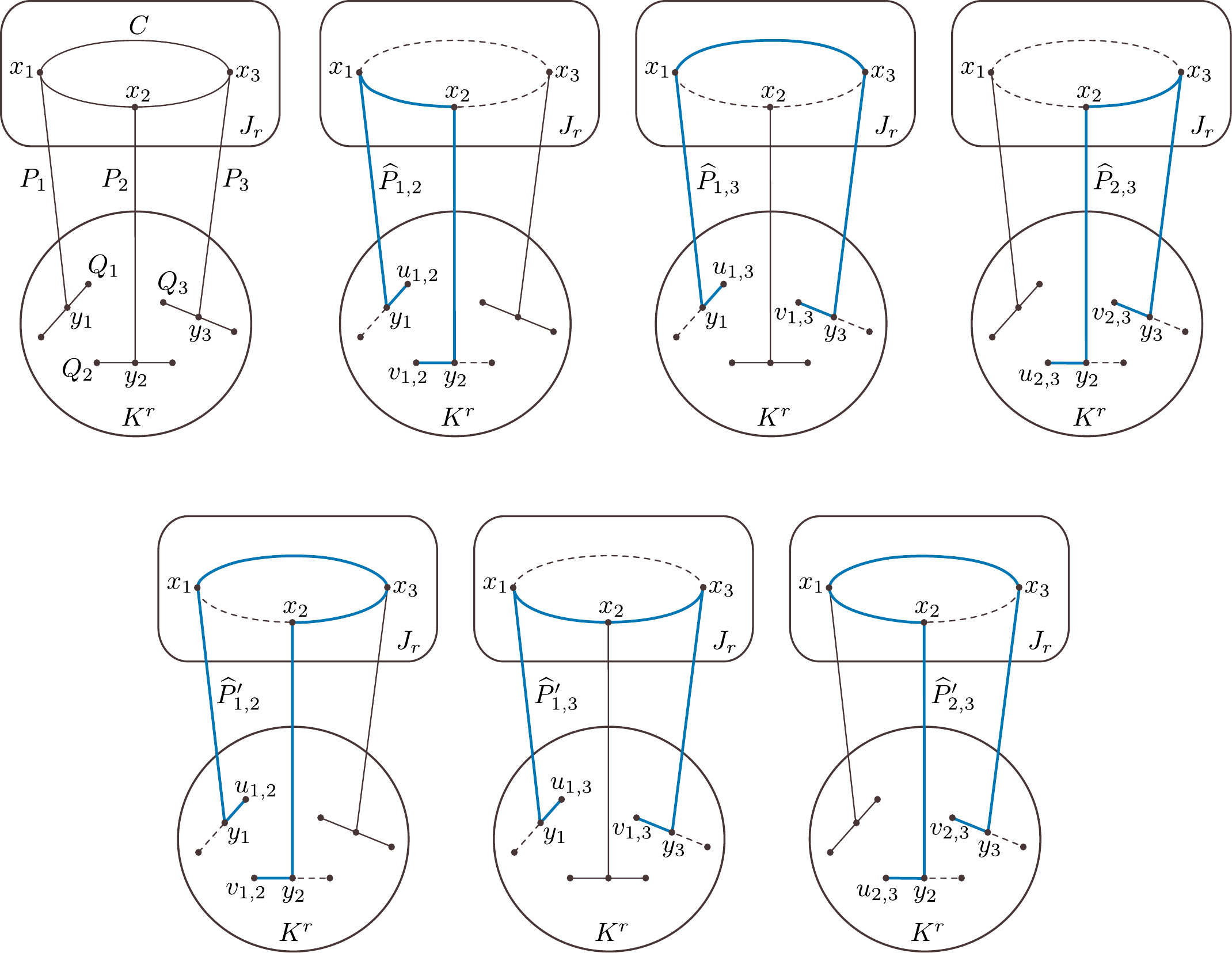}
    \caption{Construction of the paths $\w{P}_{i,j}$ and $\w{P}'_{i,j}$, for distinct $i,j\in [3]$.  } 
    \label{fig:construction}
\end{figure}
We say that an $H$-subdivision obtained after these steps is \textit{derived} from $C$, and note that each edge of $C$ is contained in exactly three $H$-subdivisions that are derived from $C$. Indeed, if $e$ lies in $P_{i,j}$, say, then $e$ is contained in $H_{i,j}$, $H'_{i,k}$ and $H_{j,k}'$. Moreover, the edges in the intersection of $H_{i,j}$, $H'_{i,k}$ and $H_{j,k}'$ are contained in $P_{i,j}\cup K^r$, hence in $C\cup K^r$.

Consider the family $\mathcal{F}:=\bigcup_{r=1}^4\mathcal{F}_r$ and note that it is $\ell$-almost-balanced, since each $H$-subdivision that it contains shares all of its branch paths with $K$, with the exception of only one branch path of the form $\w{P}_{i,j}$ or $\w{P}'_{i,j}$. We will show that $\mathcal{F}$ separates $E(G)$. Indeed, let us take arbitrary $e, e'\in E(G)$ and show that they are separated by $\mathcal F$. There are distinct $r,s\in [4]$ such that  $e$ is contained in $\mathcal{E}_r$ and $\mathcal{E}_s$. Therefore, there are elements $C^r\in\mathcal{C}_r$ and $C^s\in\mathcal C_s$ that contain $e$ but not $e'$. If either of $C^r$, $C^s$ is an edge, then it is also an element of $\mathcal F$, so $e$ is separated from $e'$. Otherwise, both $C^r$ and $C^s$ are cycles, and so the intersection of all six subdivisions derived from $C^r$ and $C^s$ and containing $e$ lies in $(C^r\cup K^r)\cap (C^s\cup K^s)$. As $e'$ is not contained in $C^r$, in $C^s$, or in $K^r\cap K^s=\emptyset$, one of these subdivisions separates $e$ from $e'$.

It remains to show that $\mathcal F$ has at most linear size. Indeed, as every element of $\mathcal{C}_r$ gives rise to at most six elements of $\mathcal{F}_r$, we have
	\[|\mathcal{F}|\le \sum_{r=1}^4 6\cdot 41|J_r|=984n.\]
\end{proof}

We note that Theorem \ref{3con} follows directly from the proof of Lemma \ref{lemma:3-connected}.

\section{Other separating systems 
}

In this section we offer some initial remarks concerning Question \ref{question:1} and related problems. 

\subsection{Bipartite separation systems}


Let $\Sigma$ be a class of graphs the members of which we mean to use to separate the edges of an arbitrary graph $G$. As discussed in the Introduction, including $K_2$ in $\Sigma$ is necessary. Noting that for every graph $G$ the edge set $E(G)$ is always a strongly separating system in $G$, we would like to understand for which classes $\Sigma$ it is possible to beat the trivial bound $O(n^2)$ for the size of a strongly separating system of $G$.  We observe that we have to restrict our attention to infinite classes of bipartite graphs (and superclasses thereof), since $K_{\frac{n}{2},\frac{n}{2}}$ has $\Omega(n^2)$ edges but only contains bipartite subgraphs. A subclass of bipartite graphs that is both simple and interesting is the class $\mathcal B$ of balanced complete bipartite graphs.

 
 The classic K\H{o}vari-S\'os-Tur\'an theorem (KST) implies that every graph can be covered with $o(n^2)$ balanced complete bipartite subgraphs. Indeed, let us recount what the KST theorem says for balanced complete bipartite graphs of logarithmic size: for every $\varepsilon>0$ there exists $c_{\varepsilon}>0$ such that every graph on $\varepsilon n^2$ edges contains the balanced complete bipartite graph on $2c_{\varepsilon}\log n$ vertices. By choosing $\varepsilon=o(1)$ so that $c_{\varepsilon}=\omega\left(\frac{1}{\log n}\right)$, we partition $E(G)$ into at most $\left(\frac{n}{c_{\varepsilon}\log n}\right)^2$ copies of said bipartite graph and $\varepsilon n^2$ edges. Such a decomposition can be produced in polynomial time \cite{mubayi2009finding}.
 
 It is natural to wonder whether a similar bound holds for edge-separation. To answer this, we will need a simple lemma:

 \begin{lemma}\label{bisep}
     The balanced complete bipartite graph $K_{n,n}$ has a separating $\mathcal B$-system of size $O(\log n)$.
 \end{lemma}

 \begin{proof}
     Let $A$ and $B$ be the two sides of $K_{n,n}$. By a standard binary search argument, there exists a family of logarithmic size, say $\mathcal F_1$, that is closed under complements and strongly separates the set $A$. We mirror $\mathcal F_1$ across to $B$ to obtain $\mathcal F_2$. For every element $F_1\in\mathcal F_1$ and its mirror $F_2\in\mathcal F_2$, we take the balanced complete bipartite graphs $K_{n,n}[F_1,F_2]$ and $K_{n,n}[F_1,F_2^c]$ and add them to the family $\mathcal F$. Then $\mathcal F$ strongly separates $E(K_{n,n})$. Indeed, given distinct edges $e$ and $e'$, their ends on one side, say $A$, must be different, say $u$ and $v$ respectively. If $F_1\in \mathcal F_1$ is a set that separates $u$ from $v$, and $F_2\in\mathcal F_2$ is the mirror of $F_1$, then the other end of $e$ is either in $F_2$ or in $F_2^c$, so either $E(K_{n,n}[F_1,F_2])$ or $E(K_{n,n}[F_1,F_2^c])$ contains $e$ but it does not contain $e'$. 
 \end{proof}

\begin{figure}
    \centering
    \includegraphics[scale=1.5]{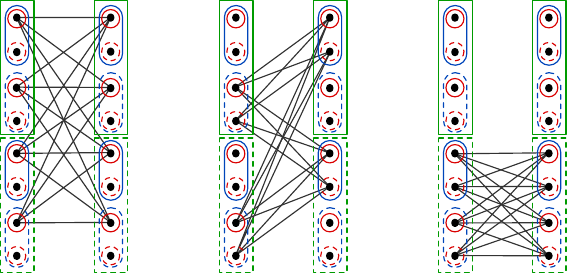}
    \caption{Construction of a separating $\mathcal B$-system for $K_{n,n}$ using subgraphs $K_{n,n}[F_1,F_2]$, where $F_1$ ranges over a family on $A$ closed under complements, and $F_2$ is either the mirror or the complement of the mirror of $F_1$ on $B$. This yields a family of size $O(\log n)$ that separates all edges.}
    \label{fig:bipartite}
\end{figure}
 We can now cover $G$ with members of $\mathcal{B}$ of logarithmic size and then we use Lemma \ref{bisep} to strongly separate the edges of the elements of said cover. We thus obtain a separating $\mathcal{B}$-system of size $O\left(\frac{\log\log n}{(c_{\varepsilon}\log n)^2}+\varepsilon\right)n^2=o(n^2)$, hence: 

 \begin{corollary}\label{completebipartitestuff}
     Every graph on $n$ vertices has a separating $\mathcal{B}$-system of size $o(n^2)$.
 \end{corollary}

 This leads us to the proof of Theorem \ref{bipartitestuff}.

 \begin{proof}[Proof of Theorem \ref{bipartitestuff}] Let $\Sigma$ be a class of connected bipartite graphs that is \emph{large}, i.e. for every large enough $n$, $\Sigma$ has a member $K_{t,s}$ with $t\leq s$ and $\omega(\log n)\leq s\leq n$. We now cover each element of the separating $\mathcal{B}$-system from Corollary~\ref{completebipartitestuff} with members of $\Sigma$. It is easy to see that, given a bipartite graph $B$ with larger part of size $s$, the edges of $K_{n,n}$ can be covered with $\frac{n^2}{s}$ copies of $B$. We therefore obtain a separating $\Sigma$-system of size $o\left(\left(\frac{(c_{\varepsilon}\log n)^2}{\log\log n}\right)\left(\frac{\log\log n}{(c_{\varepsilon}\log n)^2}\right)+\varepsilon\right) n^2=o(n^2)$.
 \end{proof}

 A fair (if not exact) reformulation of Question \ref{question:1} is then:

 \begin{question}
     Which large classes of bipartite graphs induce separating systems of linear size?
 \end{question}

\subsection{Subgraph separation}
\subsubsection{The class of all (connected) graphs}
For a graph $H$ and a class $\Sigma$, an $H$-\emph{separating} $\Sigma$-\emph{system} of a graph $G$ is a family $\mathcal F$ of subgraphs of $G$ that consists of members of $\Sigma$ and strongly separates the set of copies of $H$ in $G$, in the sense that for every two copies of $H$, say $H_1$ and $H_2$, in $G$ there exist $F_1, F_2\in \mathcal F$ such that $H_1\subseteq F_1\not\supseteq H_2$ and $H_2\subseteq F_2\not\supseteq H_1$. 

\begin{question}
    For a graph $H$ and a function $f:\mathbb{N}\rightarrow\mathbb{N}$, what are the necessary and sufficient conditions that a class $\Sigma$ must satisfy for every graph $G$ on $n$ vertices to have an $H$-separating $\Sigma$-system of size $O(f(n))$?
\end{question}

It seems prudent to assume that $H\in\Sigma$ in order to ensure that $G$ has at least some $H$-separating $\Sigma$-system, no matter the size. We begin by noting that for every graph $H$ there exists a class $\Sigma$ that yields $H$-separating $\Sigma$-systems of logarithmic size: it is the class of supergraphs of $H$. In order to see this, we employ the following theorem (for a more general form, see \cite{langi2016separation}). Given a set $X$, we call an ordered pair $(V,W)$ of disjoint subsets of $X$ a \emph{constraint} of \emph{size} $|V|+|W|$. A subset $S\subseteq X$ \emph{satisfies} the constraint $(V,W)$ if $V\subseteq S$ and $W\cap S=\emptyset$, and a family $\mathcal{F}\subseteq 2^X$ \emph{satisfies} $(V,W)$ if one of its elements does.  

\begin{theorem}[\cite{langi2016separation}]\label{langi} For any set of $N$ constraints of equal sizes there is a family of size $O(\log N)$ satisfying each of the constraints. 
\end{theorem}
Let $X=E(G)$ and consider all the constraints $(E,e)$ where $E$ is the set of edges of a copy of $H$ and  $e$ is an edge not in $E$. Then our claim follows from Theorem \ref{langi}.

What if we insist that the members of $\Sigma$ be connected? Suppose that $H$ is itself 2-connected. If $G$ is connected, then for every subgraph $S\subseteq G$ that contains one or more copies of $H$, one can find a connected supergraph $S'$ of $S$ that contains exactly the same copies of $H$, so the bound remains logarithmic. If $G$ is not connected, then we split $S$ across the connected components $\{G_i\}_{i=1}^k$ and find a connected supergraph for each part $S\cap G_i$, so that the resulting system is of size $O(\sum_{i=1}^k \log|G_i|)=O(n)$. This is tight, as $G$ may be a union of vertex-disjoint copies of $H$. 

On the other hand, if $H$ is disconnected, the connectivity requirement for $\Sigma$ cannot be fulfilled by any class: simply consider the case in which $G$ is a path and $H$ is a matching of size two. Then there is a copy of $H$ consisting of the two extremal edges of $G$, and every connected graph that contains said copy also contains all other copies.

The remaining case seems interesting.

\begin{question}
    Is it true that for every connected graph $H$ there exists a class $\Sigma$ of connected graphs such that every graph $G$ on $n$ vertices has an $H$-separating $\Sigma$-system of linear size?  
\end{question}

Of course, it is easy to see that if $\Sigma$ is the class of connected supergraphs of a given connected graph $H$, then every graph $G$ has an $H$-separating $\Sigma$-system of size at most $2n^2$, namely the family consisting of the connected components of $G-e$ that are in $\Sigma$ for each $e\in E(G)$. 

\subsubsection{Non-trivial size}

Another natural problem is to determine, for a graph $H$, which classes $\Sigma$ induce $H$-separating $\Sigma$-systems of size $o(n^{|H|})$. We believe that the answer is analogous to that for edge separation, in a sense that we explain below. Let $H$ be a fixed graph with at least one edge.

 A \emph{blowup} of $H$ is a graph obtained from it by substituting each $x\in V(H)$ with a set of vertices $V_x$ and each $xy\in E(H)$ with a complete bipartite graph between $V_x$ and $V_y$. We say that a blowup of $H$ is $\ell$-\emph{balanced} if each $V_x$ has the same size $\ell$. We may assume that $\Sigma$ is an infinite class of subgraphs of blowups of $H$ (or superclass thereof), otherwise there is no way to cover the $n/|H|$-blowup of $H$ with $o(n^{|H|})$ copies of $H$. Let $\mathcal H$ be the class of balanced blowups of $H$. By the Alon-Shikhelman Theorem \cite{alon2016many}, we know that $o(n^{|H|})$ copies of $H$ suffice to induce any particular member of $\mathcal{H}$ of any size up to $o\left((\log n)^{\frac{1}{|H|-1}}\right)$ as a subgraph of $G$\footnote{Although in \cite{alon2016many} only blowups of constant size are considered, the result of Erd\H{o}s \cite{erdos1964extremal} on which the proof of the Alon-Shikhelman Theorem is based goes up to the polylog regime.}. However, we do not know whether they also induce $o(n^{|H|})$ copies of members of $\mathcal H$ of that size (or otherwise) that cover all the copies of $H$. This is the first hurdle to overcome in order to understand $H$-separation.     

\begin{question}
    Does every graph $G$ on $n$ vertices have $o(n^{|H|})$ subgraphs that are members of $\mathcal H$ and cover every copy of $H$ in $G$? If so, can these subgraphs be taken to be of size $f(n)=polylog(n)$? 
\end{question} 

That said, we have the following result, which is simply an extension of Lemma \ref{bisep}.

\begin{theorem}
    The $n$-balanced blowup of a graph $H$ has an $H$-separating $\mathcal{H}$-system of size $O(\log n)$.
\end{theorem}

\begin{proof}

Similarly to the proof of \ref{bisep}, every set $V_x$ for $x\in V(H)$ in $H(n)$ has a family of $O(\log n)$ subsets, say $\mathcal{F}_x$, that strongly separates any subset of $V_x$ of size at most $|H|$ from any edge that is not contained in it. Let $\mathcal{F}:=H(n)[\{S_x\in \mathcal{F}_x|x\in V(H)\}]$ and let $H_1$, $H_2$ be distinct copies of $H$ in $H(n)$. In particular, let $e$ be an edge of $H_2$ but not of $H_1$.   

If $e$ lies in $H(n)[V_x,V_y]$ and the latter does not contain any edges of $H_1$, then it is clear that $H_1$ is separated from $H_2$. Otherwise, one of the ends of $e$, say $u\in V_x$, is not an end of an edge of $H_1$ that lies in $H(n)[V_x,V_y]$. Let $S_x\in \mathcal{F}_x$ be a set that separates $V(H)\cap V_x$ from $e$. Let $F$ be an element of $\mathcal{F}$ that contains $H_1$ and intersects $V_x$ at $S_x$. Then $F$ does not contain $e$, or $H_2$.   

\end{proof}

    Suppose that the answer to Question \ref{question:1} is affirmative. Suppose additionally that  each copy of $H$ is covered $o\left(\frac{f(n)^{|H|}}{\log\log n}\right)$ times. 
    Since $\ell$-blowups of $H$ contain $\ell^{|H|}$ copies of $H$, we have an $H$-separating $\mathcal{H}$-system of size \[o\left(\log\log n\cdot\frac{n^{|H|}\cdot\frac{f(n)^{|H|}}{\log\log n}}{f(n)^{|H|}}\right)=o(n^{|H|}).\] 

    We wonder whether this is indeed the case for every $n$-vertex graph $G$.

    \begin{question}
        Does every graph $G$ on $n$ vertices have an $H$-separating $\mathcal{H}$-system of size $o(n^{|H|})$?
    \end{question}






\bibliographystyle{elsarticle-num}
\bibliography{main}







\end{document}